\documentclass[12pt]{article}

\usepackage{url}

\usepackage{cite}
\usepackage{amsfonts}
\usepackage{amssymb}
\usepackage{amsmath}
\usepackage{amsthm} 

\newtheorem{definition}{Definition}

\newtheorem{theorem}{Theorem}

\newtheorem{problem}{Problem}
\newtheorem{conjecture}{Conjecture}

\def\Y{\mathbb Y}
\def\la{\lambda}
\def\Ind{\operatorname{Ind}}
\def\Id{\operatorname{Id}}
\def\si{\sigma}
\def\S{{\mathfrak S}}

\urldef\vershik\url{avershik@pdmi.ras.ru}
\urldef\natalia\url{natalia@pdmi.ras.ru}

\author{A.~M.~Vershik\thanks{%
St.~Petersburg Department of Steklov Institute of Mathematics and
St.~Petersburg State University, St.~Petersburg, Russia; Institute for Information Transmission Problems, Moscow, Russia.
E-mail: \vershik. Supported by the RFBR grant 17-01-00433.}
\and
N.~V.~Tsilevich\thanks{%
St.~Petersburg Department of Steklov Institute of Mathematics and
St.~Petersburg State University, St.~Petersburg, Russia.
E-mail: \natalia. Supported by the RFBR grant 17-01-00433.}
\and
}

\title{Groups generated by involutions of
diamond-shaped graphs, and deformations of Young's orthogonal form}

\date{October 21, 2019}

\begin{document}

\maketitle

\begin{abstract}
With an arbitrary finite graph having a special form of 2-intervals (a
diamond-shaped graph) we associate a subgroup of a symmetric group and
a representation of this subgroup; state a series of problems on such
groups and their representations; and present results of some computer
simulations. The case we are most interested in is that of the Young
graph and subgroups generated by natural involutions of Young
tableaux. In particular, the classical Young's orthogonal form can be
regarded as a deformation of our construction. We also state
asymptotic problems for infinite groups.

\end{abstract}

\section{The main construction}

\subsection{Combinatorial involutions on diamond-shaped graphs and the group of permutations of paths}
We begin with a finite directed graded graph\footnote{For the general theory of graded graphs and Bratteli diagrams, see, e.g.,~\cite{Itogi} or~\cite{Kerov}.} with one minimal and one maximal element, and assume that it also satisfies the following property.

\begin{definition}
A graded graph is said to be \emph{diamond-shaped} if every its nonempty
$2$\nobreakdash-interval\footnote{By a $2$-interval $[v,w]$ of a graded graph~$\Gamma$ we mean a subgraph of~$\Gamma$ consisting of two vertices $v$~and~$w$ of levels $k$~and~$k+2$, respectively (for some~$k$), and all the vertices of level~$k+1$ connected with both these vertices.}
contains either one or two vertices of the intermediate level, i.e., is either a chain or a rhombus.
\end{definition}

Diamond-shaped graphs appear in the following situation. Consider a finite partially ordered set~$P$ with minimal element~$\varnothing$ and the partially ordered set~$J(P)$ of its  ideals (subsets that contain with every element all smaller elements). This partially ordered set of ideals is a distributive lattice, and its Hasse diagram is a diamond-shaped graph. By a well-known theorem, the converse is also true: every finite distributive lattice is the lattice of  ideals of a finite partially ordered set (see, e.g., \cite[Chap.~3]{Stanley1} or~\cite{Birk}). The property of being diamond-shaped is a weakening of distributivity; the latter imposes conditions on all intervals (and not only $2$-intervals).

Consider an arbitrary ${\mathbb Z}_+$-graded finite diamond-shaped graph~$\Gamma$; denote by~$\Gamma_k$ the set of its vertices of level~$k$.
Let $T(\Gamma)$ be the set of maximal paths of~$\Gamma$, i.e., paths connecting the minimal vertex (of level~$0$) with the maximal vertex (of level~$n$). Denote by~${\mathfrak S}_{\Gamma}$ the group of all permutations of the set~$T(\Gamma)$ (clearly, it is isomorphic to the symmetric group~${\mathfrak S}_N$ where $N$~is the total number of maximal paths in the graph).

\begin{definition}\label{def:combinv}
For  $i=1,2, \dots, n-1$, the \emph{combinatorial involution} $\sigma_i$
is the involution $\sigma_i\in{\mathfrak S}_{\Gamma}$ that acts as follows. Let $t=(t_0,t_1, \dots, t_n)\in T(\Gamma)$ where $t_k\in\Gamma_k$. The involution~$\sigma_i$ leaves all vertices of the path~$t$ except~$t_{i+1}$ unchanged. Consider the
$2$\nobreakdash-interval $[t_i,t_{i+2}]$ in~$\Gamma$.  If it is a chain, then $\sigma_i(t)=t$. If it is a rhombus with intermediate vertices~$t_{i+1}, t'_{i+1}$, then
$$
\sigma_i(t)=(t_0,\dots,t_i,t'_{i+1},t_{i+2}, \dots, t_n).
$$
\end{definition}

By our assumptions on the graph, the action of~$\sigma_i$ is well defined on all paths of the graph.

\begin{definition}
The \emph{group of permutations of paths} of the graph~$\Gamma$ is the group~$G_{\Gamma}=\langle\sigma_1,\dots,\sigma_{n-1}\rangle$ generated by the $n-1$~combinatorial involutions $\sigma_1, \dots, \sigma_{n-1}$
of~$\Gamma$.
 \end{definition}

Consider the $\Bbb R$-vector space~$V(\Gamma)$ of formal linear combinations of maximal paths of~$\Gamma$. Above we have defined not only the group~$G_{\Gamma}$, but also a~representation of this group in~$V(\Gamma)$. The action of every involution~$\sigma_i$ is not identical only \emph{in two-dimensional spaces corresponding to the pair of paths of the same rhombus at the levels~$i$, $i+1$,~$i+2$}.

The obvious relations satisfied by the involutions~$\sigma_i$, $i=1,2,\dots, n-1$, are as follows:
$$\sigma_i^2=\operatorname{id}, \qquad \sigma_i\sigma_j=\sigma_j\sigma_i  \quad\mbox{for}\quad |i-j|\geq 2,
$$
where $\operatorname{id}$ is the identity transformation.

The relations between $\sigma_i$ and $\sigma_{i+1}$ depend on the graph in a complicated manner, and it is difficult to describe them in the general case. Thus, $G_\Gamma$~is defined as a permutation group (a subgroup of a symmetric group) by its generators. (For the general theory of permutation groups, see, e.g.,~\cite{DM}.)

Our main question is how does this group look like for various graphs, and whether it is possible to obtain a classification of such groups.

\begin{problem}
To what extent does the group $G_{\Gamma}$ characterize the original graph~$\Gamma$? How can one characterize the class of graphs corresponding to isomorphic groups~$G_{\Gamma}$?
\end{problem}

Consider the first simple example. Hereafter, by~${\mathfrak S}_n$ we denote the symmetric group of degree~$n$.

\begin{theorem}Let $\Gamma$ be the Hasse diagram of the finite Boolean algebra~$B_n$ with $n$~atoms.
Then $G_{\Gamma}\simeq{\mathfrak S}_n$.
\end{theorem}

\begin{proof}
It is sufficient (and easy) to verify that in this case we have the relations $\sigma_i\sigma_{i+1}\sigma_i=\sigma_{i+1}\sigma_i\sigma_{i+1}$, $i=1,2, \dots, {n-1}$, which, together with the above ones, define the group~${\mathfrak S}_n$. Note that the paths in~$\Gamma$ can be indexed in a natural way by the permutations of the numbers $1,\dots,n$, and the representation of the group~$G_{\Gamma}$ in the space~$V(\Gamma)$ is the regular representation of the symmetric group~${\mathfrak S}_n$.
\end{proof}

The next result is less obvious. By definition, the $d$-dimensional Pascal graph~$P^d$, where $d\ge2$, is the ${\mathbb Z}_+$-graded graph whose $n$th level consists of all $d$-tuples $(k_1,\ldots,k_d)\in{\mathbb Z}_+^d$ such that $k_1+\ldots+k_d=n$, and an edge connects two vertices of neighboring levels that are obtained from each other by changing one coordinate by~$1$. In particular,  $P:=P^2$ is the ordinary Pascal graph (infinite Pascal triangle), see, e.g.,~ \cite{VPascal}.
For an arbitrary vertex $v\in P^d$, denote by~$P^d(v)$ the finite subgraph in~$P^d$ induced by the set of vertices of all paths leading from the minimal vertex $\varnothing=(0,\dots,0)$ to~$v$. We say that $P^d(v)$~is an interval of length~$n$ of the Pascal graph if $v$~is a vertex of level~$n$.

\begin{theorem}\label{th:pascal}
If $\Gamma$  is an arbitrary interval of length~$n$ of the $d$-dimensional Pascal graph, then
${G_{\Gamma}={\mathfrak S}_{n}}$.
\end{theorem}

\begin{proof}
Let $\Gamma=P^d(v)$ where $v\in P^d_n$ is a vertex of level~$n$. Let $v=(m_1,\dots,m_d)$. Then the vertices of the graph~$P^d(v)$ can be indexed in an obvious way by all sequences of the form $(a_1,\dots,a_n)$ where $a_i\in[d]=\{1,\dots,d\}$ and $\#\{i:a_i=k\}=m_k$ for every $k=1,\dots,d$. It is easy to see that, under this parametrization, the combinatorial involution~$\sigma_i$, $i=1,\dots,n-1$, acts  as the transposition that swaps
 $a_i$~and~$a_{i+1}$. But this action coincides with the action of the Coxeter generators $s_1,\dots,s_{n-1}$ of the symmetric group~${\mathfrak S}_n$ in the standard realization of the induced representation
$\Ind_{{\mathfrak S}_{m_1}\times\dots\times{\mathfrak S}_{m_d}}^{{\mathfrak S}_n}\Id$, where $\Id=\pi_{(n)}$~is the identity representation of~${\mathfrak S}_n$. This implies the desired result.
\end{proof}

Note that we have described, in particular, the representation of the group~$G_\Gamma$ arising in this case; it is the representation $\Ind_{{\mathfrak S}_{m_1}\times\dots\times{\mathfrak S}_{m_d}}^{{\mathfrak S}_n}\Id$ of the symmetric group induced from an appropriate Young subgroup.

In Sec.~\ref{sec:young}, we will consider groups of permutations of paths on the Young graph. As we will see, in this case $G_{\Gamma}$~is not always isomorphic to a~symmetric group, and even in the cases when it is isomorphic to a symmetric group, this group, in contrast to the previous examples, is defined as being generated not by the traditional transpositions, but by involutions of general form. Apparently, in the general case the groups of the type~$G_{\Gamma}$ constitute a quite special class of groups.

\subsection{Problems about infinite groups}
Obviously, the group corresponding to an interval of a graph is a natural subgroup of the group corresponding to the graph itself. Hence, we can consider an infinite path in an infinite graded graph (Bratteli diagram), for instance, in the Young graph, and the inductive limit of the groups of permutations of paths corresponding to initial segments of this path. A number of natural questions arise.

\begin{problem}
How does the resulting infinite group look like? For what pairs of paths are the corresponding groups  naturally isomorphic?
\end{problem}

Even for the Young graph~$\Y$ (see Sec.~\ref{sec:young}), this results in a series of especially intriguing problems. Consider an indecomposable character of the infinite symmetric group~${\frak S}_{\infty}$ and the corresponding representation (see, e.g., \cite{VK, Kerov}). For almost every, with respect to the corresponding measure, path $T=(\la_0,\la_1,\dots)$, where $\la_k\in\Y_k$,  consider the group~$G_T$ which is the inductive limit of the groups~$G_{\Y_{\la_n}}$ corresponding to finite initial segments of the path. Clearly, finitely equivalent paths lead to the same group.

\begin{problem}Are all these infinite groups~$G_T$ isomorphic for almost all paths? If no, what is the partition of the set of paths into isomorphism classes of the corresponding groups? If yes, is this unique (up to isomorphism) group isomorphic to the infinite symmetric group?
\end{problem}

\section{Groups generated by combinatorial involutions on the Young graph}\label{sec:young}

In this section, we describe the above setting  in detail for the Young graph. Recall that by~${\mathfrak S}_n$ we denote the symmetric group of degree~$n$. By $s_i\in{\mathfrak S}_n$ we denote the Coxeter generator~$s_i=(i,i\!+\!1)$, ${i=\!1,\ldots,n-1}$.

Let $\Y$ be the Young graph; the $n$th level of~$\Y$ is the set~$\Y_n$  of Young diagrams with $n$~cells, and an edge connects two vertices of neighboring levels such that the larger one is obtained from the smaller one by adding one cell. Given a diagram $\la\in\Y_n$, denote by $\Y_\la$ the finite subgraph in~$\Y$ induced by the set of vertices of all paths leading from the minimal vertex (empty diagram) to~$\lambda$. Then the set of paths in the graph~$\Y_\la$ can be identified with the set of standard Young tableaux of shape~$\la$,  and the linear space $V_\la:=V(\Y_\la)$ spanned by~$T_\la$ is the space of the irreducible representation~$\pi_\la$ of the group~${\mathfrak S}_n$ corresponding to the diagram~$\la$.

Recall that the action of the Coxeter generators~$\si_i$ in the representation~$\pi_\la$ in the Gelfand--Tsetlin basis indexed by the Young tableaux of shape~$\la$ is given by Young's orthogonal form (see Sec.~\ref{sec:orth}). But, according to Definition~\ref{def:combinv}, we consider
the following {\it permutation} action of the Coxeter generators $\sigma_1,\dots,\sigma_{n-1}$:
$$
\rho_\la(\si_i)t=\begin{cases}
t_i' & \text{if $i$ and $i+1$ lie in different columns and different rows of~$t$},\\
t& \text{otherwise},
\end{cases}
$$
where $t_i'$ is the standard tableau obtained from~$t$ by swapping $i$~and~$i+1$. Observe that $\sigma_1$ is always the identity transformation, since $1$ and $2$ always lie either in one column, or in one row.

Consider the group $G_\la=\langle\sigma_2,\dots,\sigma_{n-1}\rangle$ of permutations of the graph~$\Y_\la$ generated by the involutions~$\sigma_i$, $i=2,\ldots,n-1$. Our goal is to study the groups~$G_\la$. Obviously, $G_\la$~is a subgroup of the total group of permutations of the set~$T_\la$ of paths in~$\Y$ leading to the vertex~$\la$, i.e., of the symmetric group~${\mathfrak S}_{\dim\la}$, where $\dim\la$ is the dimension of the diagram~$\la$, i.e., the number of such paths.

\subsection{Exact results}

\begin{theorem}\label{th:hooks}
If $\la=(n-k,1^k)$ is a hook diagram, then $G_\la$ is isomorphic to the symmetric group~${\mathfrak S}_{n-1}$.
\end{theorem}

\begin{proof}
It is easy to see that for a hook diagram ${\lambda=(n-k,1^k)}$, the graph~$\Y_\lambda$ is isomorphic to the finite interval~$P(v)$ of the Pascal graph where $v={(n-k-1,k)}$, hence the result follows from Theorem~\ref{th:pascal}.
\end{proof}

In particular, as follows from Theorem~\ref{th:pascal}, in this case we have the representation
$\Ind_{{\mathfrak S}_k\times{\mathfrak S}_{n-k-1}}^{{\mathfrak S}_{n-1}}\Id$ of the symmetric group $G_\la\simeq{\mathfrak S}_{n-1}$ in the space~$V_\la$.

\begin{theorem}
If $\la=(n-2,2)$ for $n\ge4$, then $G_\la$ is isomorphic to the symmetric group~${\mathfrak S}_{\dim\la}$.
\end{theorem}

\begin{proof}
Obviously, the group $G_{(2,2)}$ is isomorphic to $G_{(2,1)}$, so the corresponding result follows from Theorem~\ref{th:hooks}; thus, in what follows we assume that $n\ge5$. Denote $G:=G_\la$.

Let $T_1$ and $T_2$ be the sets of Young tableaux of shape~$\la$ in which the element~$2$ lies in the first and second row, respectively. Obviously, $T=T_1\cup T_2$ and $T_1,T_2$ are invariant under $\si_i$ for $i\ge 3$. For $k=1,2$, denote by $\si_i^{(k)}$ the restriction of $\si_i$ to $T_k$, and let $G_k=\langle\si_3^{(k)},\ldots,\si_{n-1}^{(k)}\rangle$.
It is easy to see that the action of the operators $\si_3^{(1)},\ldots,\si_{n-1}^{(1)}$  coincides with the action of the Coxeter generators $s_1,\ldots,s_{n-3}$ of~$\S_{n-2}$ in the representation
$\Ind_{\S_2\times\S_{n-4}}^{\S_{n-2}}\Id$, so $G_1$ is isomorphic to $\S_{n-2}$. Analogously, the action of the operators $\si_4^{(2)},\ldots,\si_{n-1}^{(2)}$ coincides with the action of the Coxeter generators $s_1,\ldots,s_{n-4}$ of~$\S_{n-3}$ in the representation
$\Ind_{\S_1\times\S_{n-4}}^{\S_{n-3}}\Id$ (while, obviously, $\si_3^{(2)}$ is the identity operator), so $G_2$ is isomorphic to $\S_{n-3}$.

Now, let us show that these two actions are independent, in the sense that the subgroup $G':=\langle \si_3,\ldots,\si_{n-1}\rangle\subset G$  is isomorphic to $\S_{n-2}\times\S_{n-3}$. Let $\tau=\si_3\si_4\ldots\si_{n-1}$. It follows from the observations in the previous paragraph that $\tau_1:=\tau|_{T_1}$ can be identified with the cycle $(1,2,\ldots, n-2)$ in~$\S_{n-2}$, while $\tau_2:=\tau|_{T_2}$ can be identified with a cycle of length~$n-3$ in $\S_{n-3}$. Therefore,  $\tau^{n-3}$ acts as the cycle $(n-2,n-3,\ldots,1)$ on $T_1$ and identically on~$T_2$. Further, $\si_3$ acts as the transposition~$(1,2)$ on~$T_1$ and identically on~$T_2$. But it is well known that the permutations $(n-2,n-3,\ldots,1)$ and $(1,2)$ generate the symmetric group~$\S_{n-2}$. Hence, $G'$ contains $\S_{n-2}\times\{e\}$, which clearly implies the desired claim.

It remains to show that $\langle G',\si_2\rangle$ is isomorphic to $\S_{\dim\la}$. Let $T_{11}$ be the subset of~$T_1$ consisting of the tableaux in which the element $3$ lies in the first row and $T_{12}$ be the subset of~$T_1$ consisting of the tableaux in which the element $3$ lies in the second row. Obviously, there is a natural bijection between $T_{12}$ and $T_2$ which identifies tableaux differing only by the positions of the elements~$2$ and~$3$. Let $T_2=\{t_1,\ldots,t_k\}$ and $T_{12}=\{t_1',\ldots,t_k'\}$, where $t_i$ and $t_i'$ correspond to each other under this bijection, and $T_{11}=\{r_1,\ldots,r_\ell\}$. Then, obviously, $\si_2$ exchanges $t_i$ and $t_i'$ for each~$i$, while leaving each $r_j$ fixed. Now, as we have already proved, there exist a permutation $g_1\in G'$ that acts as the cycle $(t_1,t_2,\ldots,t_k)$ on~$T_2$ and identically on~$T_1$ and a permutation $g_2\in G'$ that acts as the cycle $(r_1,r_2,\ldots,r_\ell)$ on~$T_1$ and identically on~$T_2$. It is easy to see that the permutation $g_1\si_2g_2$ is the cycle
$$
c=(t_1,t_2',t_2,t_3',t_3,\ldots,t_k',t_k,r_1,r_2,\ldots,r_\ell,t_1')
$$
of length~$\dim_\la$. On the other hand, there exists $g_3\in G'$ that acts as the transposition $(r_1,r_2)$ on $T_1$ and identically on~$T_2$. Then $\langle c,g_3\rangle=\S_{\dim\la}$, as required.
\end{proof}

\section{Computational results for $n\le9$}

Using the \textbf{\textit{SageMath}} software system, we have computed the orders of the groups~$G_\la$ for
diagrams~$\la$ with $n\le9$ cells. The results are as follows. 

\begin{itemize}
\item If $\la$ is a hook diagram, then $G_\la$ is isomorphic to the symmetric group~$\S_{n-1}$ (Theorem~\ref{th:hooks}).

\item If $\la$ is one of the diagrams $(4,2^2)$, $(6,3)$, $(4^2,1)$, then the order of~$G_\la$ is equal to $\frac{(\dim\la)!}2$.

\begin{conjecture} In these cases, $G_\la$ is isomorphic to the alternating group~$A_{\dim\la}$.
\end{conjecture}

\item If $\la$ is one of the diagrams $(3,2,1)$, $(4, 2,1^2)$, $(3^2,2)$, $(3^3)$, then the order of~$G_\la$ is equal to $2^{k-1}k!$ where $k=\frac{\dim\la}2$ is half the dimension of~$\la$.

\begin{conjecture} In these cases, $G_\la$ is isomorphic to the Coxeter group~$D_k$.
\end{conjecture}

{\bf Remark.} The conjecture is verified using \textbf{\textit{SageMath}} for the diagram $\la=(3,2,1)$.

\begin{conjecture} If $\la$ is a symmetric diagram that is not a hook, then $G_\la$ is isomorphic to the Coxeter group~$D_k$ where $k=\frac{\dim\la}2$.
\end{conjecture}

\item In all the other cases, the order of $G_\la$ is equal to~$(\dim\la)!$, that is, $G_\la$ is isomorphic to the total symmetric group~$\S_{\dim\la}$.
\end{itemize}

These examples indicate that the groups~$G_{\lambda}$  for Young diagrams either are Coxeter groups, or differ little from them. For other graphs, no experiments have been done, it is not even clear in which cases these groups are, for example, $2$-transitive. This question is especially interesting for groups corresponding to distributive lattices, in particular, for skew Young diagrams.

\section{Young's orthogonal form as a deformation of a representation of~$G_{\Y_\la}$}\label{sec:orth}

Let us extend our definition of groups acting in the spaces~$V_\la$ to include the classical Young's construction that realizes irreducible representations of the symmetric group in the same space~$V_\la$, but in which transpositions act not as permutations of paths, but as two-dimensional rotations with reflection. To simplify matters, we state Young's formulas in $\mathbb C$-form.

We still consider the space~$V_\la=V(\Y_{\lambda})$, where $\lambda$~is a Young diagram with $n$~cells and $\Y_{\lambda}$~is the corresponding subgraph of the Young graph. The corresponding irreducible representation~$\pi_{\lambda}$ of the group~${\mathfrak S}_n$ acts in the space~$V_\la$, and Young found formulas for the action of the Coxeter transpositions  (see, e.g., a modern exposition in~\cite{OV} or a classical one in~\cite{JK}). Namely, in each two-dimensional space corresponding to a diamond-shaped $2$-interval, we introduce a structure of the complex line~$\Bbb C$. Then the combinatorial involution introduced above turns into the transformation
 $z\mapsto i\bar z$ (which swaps the real and imaginary axes). But Young's involution corresponding to the Coxeter transposition~$\sigma_k$,  $k=1,2,\dots, n-1$, acts in each one-dimensional complex subspace of this form as
 $$z\mapsto e^{i\alpha_{\lambda}(k)}\bar z,$$
where $\alpha_{\lambda}(k)=\arctan{\sqrt{r^2-1}}$ and $r$~is the $l^1$-distance between the cells of the tableau under consideration containing the elements $k$~and~$k+1$ (the so-called axial distance). If the $2$-interval is not a rhombus but a chain, then in the corresponding one-dimensional real subspace the involution acts as~$\pm 1$ depending on whether the elements $k$~and~$k+1$ lie in the same row or in the same column.

 We see that Young's involution is a deformation of the combinatorial involution, the parameters of this deformation being real numbers~$\alpha_{\lambda}(k)$, and it is these parameters that define a subgroup of the group of unitary operators in the space~$V_{\lambda}$.

 Clearly, the finiteness of this group is related to very special values of the parameters.

\begin{problem}
For what parameters is this group of unitary operators finite? infinite?
 \end{problem}

The remarkable classical result saying that for Young's values of the parameters the group is canonically isomorphic to the symmetric group  is not quite obvious.

It is well known that the free groups generated by more than two involutions are ``wild,'' i.e., the space of their irreducible complex representations is unmanageable and has no natural parametrization. This applies not only to free groups, but also to many infinite groups generated by several involutions. But we define a group together with a fixed finite-dimensional representation (not necessarily irreducible) of this group associated with the graph. Hence it is natural to consider the following question.

\begin{problem}
For a given group of the type under consideration, characterize intrinsically the representations arising in the described construction.
 \end{problem}

  The authors are grateful to M.~A.~Vsemirnov for useful discussions.

 \end{document}